\documentclass[12pt,a4paper]{amsart}
\usepackage[utf8]{inputenc}
\usepackage{amsmath}
\usepackage{amsthm}
\usepackage{amssymb}
\usepackage{wasysym}
\usepackage{graphicx}
\usepackage{tikz}
\usetikzlibrary{positioning}
\usepackage{geometry}
\usepackage{xcolor}
\usepackage[unicode]{hyperref}
\hypersetup{
	colorlinks,
    linkcolor={red!60!black},
    citecolor={green!60!black},
    urlcolor={blue!60!black},
}
\usepackage[abbrev, msc-links]{amsrefs}
\def\N{\mathbb N}
\def\X{\mathcal X}
\def\Y{\mathcal Y}
\def\G{\mathcal G}
\def\B{\mathcal B}
\def\K{\mathcal K}

\DeclareMathOperator{\Forb}{Forb}

\newtheorem{theorem}{Theorem}[section]
\newtheorem{definition}[theorem]{Definition}
\newtheorem{lemma}[theorem]{Lemma}
\newtheorem{corollary}[theorem]{Corollary}

\newtheorem{question}[theorem]{Question}

\title{Universal Graphs for the Topological Minor Relation}
\author{Thilo Krill}
\address{Universit\"at Hamburg, Department of Mathematics, Bundesstrasse 55 (Geomatikum), 20146 Hamburg, Germany}
\email{thilo.krill@uni-hamburg.de}

\begin{document}

\begin{abstract}
A subgraph-universal graph/a topological minor-universal graph in a class of graphs $\mathcal{G}$ is a graph in $\mathcal{G}$ which contains every graph in $\mathcal{G}$ as a subgraph/topological minor. We prove that the class $\mathcal{P}$ of all countable planar graphs does not contain a topological minor-universal graph. This answers a question of Diestel and Kühn and strengthens a result of Pach stating that there is no subgraph-universal graph in $\mathcal{P}$. Furthermore, we characterise for which subdivided stars $T$ there is a topological minor-universal graph in the class of all countable $T$-free graphs.
\end{abstract}

\maketitle

\section{Introduction}\label{introduction}

Let $R$ be either the subgraph relation, the topological minor relation, or the minor relation. We say that a graph $\Gamma$ is \emph{$R$-universal} in a class of graphs $\G$ if $\Gamma\in\G$ and $\Gamma$ contains every graph $G\in\G$ with respect to $R$. In this paper, we investigate the existence of topological minor-universal graphs. All graphs here are countable and we do not distinguish between isomorphic graphs.

Ulam asked whether there exists a subgraph-universal planar graph (cf.\ \cite{pach}), which was answered negatively by Pach:

\begin{theorem}[\cite{pach}]\label{planar-leq}
There is no subgraph-universal graph in the class of all planar graphs.
\end{theorem}

There have been different approaches on how to weaken Ulam's question in order to obtain a positive answer. For example, Huynh, Mohar, \v{S}ámal, Thomassen and Wood recently constructed a graph which contains all planar graphs as subgraphs and is not planar itself, but satisfies a number of properties which are also satisfied by planar graphs \cite{huynh}. For another approach, Diestel and Kühn showed that there is a minor-universal graph in the class of all planar graphs and asked whether the same is true for topological minor-universality \cites{diestelkuhn}. We answer their question negatively, strengthening Theorem \ref{planar-leq}:

\begin{theorem}\label{planar-unlhd}
There is no topological minor-universal graph in the class of all planar graphs.
\end{theorem}

Lehner proved Theorem \ref{planar-unlhd} independently in \cite{lehner}. He showed that every graph $G$ which contains every planar graph as a topological minor contains arbitrarily large finite cliques as topological minors as well as an infinite clique as a minor. This implies Theorem \ref{planar-unlhd} as such a graph $G$ clearly cannot be planar. At the same time, Lehner's result strengthens a recent result by Huynh et al., who showed the same under the stronger premise that $G$ contains all planar graphs as subgraphs \cite{huynh}.

We generalise Theorem \ref{planar-unlhd} in a different direction than Lehner:

\begin{theorem}\label{excludingtopminors_informal}
Any class of graphs which is defined by excluding finite topological minors of minimum degree at least 3 contains a topological minor-universal graph if and only if it contains a subgraph-universal graph.
\end{theorem}

As an application, we strengthen several known results on the non-existence of subgraph-universal graphs.
In particular, Theorem \ref{planar-unlhd} is a direct consequence of Theorem \ref{planar-leq} and Theorem \ref{excludingtopminors_informal} since the countable planar graphs are precisely the countable graphs without $K_{3,3}$ and $K^5$ as topological minors.

In Sections \ref{section-pos} and \ref{section-neg}, we look at topological minor-universal graphs with forbidden subdivided stars. For a graph $X$, let $\Forb(X)$ be the class of all graphs which do not contain $X$ as a subgraph. To decide for which graphs $X$ there is an induced subgraph-universal graph or a subgraph-universal graph in $\Forb(X)$ is an ongoing quest which has been solved for many but not for all finite connected graphs $X$ (see for example \cites{furedikomjath, cherlinshelah, cherlintallgren, cherlinshelah2016}).

The question whether there exists a topological minor-universal graph in $\Forb(X)$ will be solved in this paper for all finite connected graphs $X$. However, it is only interesting if $X$ is a subdivided star for the following reasons. Firstly, it is trivial to find a topological minor-universal graph in $\Forb(X)$ if $X$ is not a subdivided star: If $X$ contains a cycle, say of length $k$, then the infinite clique in which every edge is subdivided $k$ times is topological minor-universal in $\Forb(X)$. Otherwise $X$ is a tree but not a subdivided star and thus $X$ contains two vertices of degree at least three. When we denote their distance in $X$ by $k$, the same graph as above is topological minor-universal in $\Forb(X)$.

Secondly, the existence of a topological minor-universal graph $\Gamma$ in a class of graphs $\G$ is especially interesting if $\G$ is closed under topological minors because then $\Gamma$ yields a characterisation for the graphs in $\G$: Any graph $G$ is contained in $\G$ if and only if $G$ is a topological minor of $\Gamma$. Finally, note that $\Forb(X)$ is indeed closed under topological minors for any subdivided star $X$ (in fact, excluding a subdivided star as a subgraph or as a topological minor results in the same). On the other hand, $\Forb(X)$ is not closed under topological minors for any other finite connected graphs $X$ since the universal graph in $\Forb(X)$ which we found above contains $X$ as a topological minor but $X\notin\Forb(X)$.

We characterise the existence of topological minor-universal graphs with forbidden subdivided stars as follows:

\begin{theorem}\label{forb-star}
Let $T$ be a finite subdivided star. There is a topological minor-universal graph in $\Forb(T)$ if and only if at most two of the original edges of $T$ are subdivided.
\end{theorem}

A comparison to the following result by Cherlin and Shelah shows that there are fewer positive results if we look at induced subgraph- or subgraph-universal graphs, in particular the condition on the minimum degree in Theorem \ref{excludingtopminors_informal} cannot be omitted:

\begin{theorem}[\cite{cherlinshelah}]\label{forb-tree}
Let $T$ be a finite tree. There is an induced subgraph-universal graph in $\Forb(T)$ if and only if there is a subgraph-universal graph in $\Forb(T)$ if and only if $T$ is either a path or consists of a path with an adjoined edge.
\end{theorem}

For further research on topological minor-universal graphs, we suggest the following question:

\begin{question}
For which finite connected graphs $X$ is there a topological minor-universal graph in the class of all graphs without $X$ as a topological minor?
\end{question}

Both Theorem \ref{excludingtopminors_informal} and Theorem \ref{forb-star} provide a partial answer.

\section{Preliminaries}\label{section-preliminaries}

We repeat that all graphs in this paper are countable but from now on we do distinguish between isomorphic graphs. Let $G$ and $H$ be graphs. An \emph{embedding of $G$ in $H$} is an injective map $\gamma:V(G)\to V(H)$ that preserves adjacency. A \emph{topological embedding of $G$ in $H$} is an injective map $\gamma:V(G)\to V(H)$ such that for every edge $vw\in E(G)$ there exists a $\gamma(v)$--$\gamma(w)$ path $P^{vw}$ in $H$ with the following property: For all $e\in E(G)$, the path $P^e$ has no inner vertices in the image of $\gamma$ or in any other path $P^f$ with $e\neq f\in E(G)$. If there exists an embedding of $G$ in $H$, we write $G\leq H$ and if there exists a topological embedding of $G$ in $H$, we write $G\unlhd H$. If $\gamma:v\mapsto v$ is an embedding of $G$ in $H$ (i.e. if $G$ is a subgraph of $H$), we write $G\subseteq H$. We have $G\leq H$ if and only if $H$ contains a subgraph isomorphic to $G$ and $G\unlhd H$ if and only there is a subdivision of $G$ which is isomorphic to a subgraph of $H$. The graph $H$ is a \emph{model} of $G$ if there is a partition $\{V_z:z\in V(G)\}$ of $V(H)$ into non-empty connected sets such that for all $y\neq z\in V(G)$ there is a $V_y$--$V_z$ edge in $H$ if and only if $yz\in E(G)$. We call the sets $V_z$ \emph{branch sets}. We say that $G$ is a \emph{minor} of $H$ and write $G\preccurlyeq H$ if $H$ has a subgraph which is a model of $G$.

Let $\G$ be a class of graphs and $\sqsubseteq$ a graph relation, e.g. $\sqsubseteq\in\{\leq,\unlhd,\preccurlyeq\}$. We say that a graph $\Gamma$ is \emph{$\sqsubseteq$-universal in $\G$} if $\Gamma\in\G$ and $G\sqsubseteq\Gamma$ for all $G\in\G$. We denote by $\Forb(\G;\sqsubseteq)$ the class of all graphs $H$ such that there is no graph $G\in\G$ with $G\sqsubseteq H$. If $\G=\{G_1,\ldots,G_n\}$ is finite, we also write $\Forb(G_1,\ldots,G_n;\sqsubseteq)$ for $\Forb(\G;\sqsubseteq)$. The graphs in $\Forb(\G;\leq)$ will be called \emph{$\G$-free}.

Let $G$ be a graph and $v\in V(G)$. We write $S_G(v)$ for the subgraph of $G$ with vertex set $\{v\}\cup N_G(v)$ which contains precisely all edges between $v$ and its neighbours. We write $P_n$ for a fixed path of length $n$ and $C_n$ for a fixed cycle of length $n$. Next, let $P$ be a nontrivial path in $G$ and $X$ a subset of $V(G)$. We say that $P$ is an \emph{$X$-path} if $X$ contains both endvertices of $P$ but no inner vertices of $P$.

An \emph{$\alpha$-colouring} of a graph $G$ is a map $c:V(G)\to\alpha$, where $\alpha$ is an ordinal (we will always have either $\alpha=2$ or $\alpha=\omega$). We will not explicitly mention the map $c$ and refer to $c(v)$ as the \emph{colour of $v$ in $G$}. If we consider a graph $G$ together with an $\alpha$-colouring of $G$, we say that $G$ is an \emph{$\alpha$-graph}. Let $G$ and $H$ be $\alpha$-graphs. A \emph{(topological) $\alpha$-embedding of $G$ in $H$} is a (topological) embedding of $G$ in $H$ which additionally preserves the colour of vertices. We write $G\leq_\alpha H$ if there is an $\alpha$-embedding of $G$ in $H$ and $G\unlhd_\alpha H$ if there is a topological $\alpha$-embedding of $G$ in $H$. We will also use the notions $\Forb(\X;\sqsubseteq)$ and $\sqsubseteq$-universality for $\sqsubseteq\in\{\leq_\alpha,\unlhd_\alpha\}$; all graphs that are involved in their definitions have to be $\alpha$-graphs.

\section{Forbidden Topological Minors and Forbidden  Minors}\label{section-topological}

\begin{theorem}\label{excludingtopminors}
If $\X$ is a class of finite graphs with minimum degree at least 3, then $\G:=\Forb(\X;\unlhd)$ contains a $\unlhd$-universal graph if and only if $\G$ contains a $\leq$-universal graph.
\end{theorem}

\begin{proof}
Clearly, a $\leq$-universal graph in $\G$ is also $\unlhd$-universal. Conversely, suppose that there is a $\unlhd$-universal graph $\Gamma$ in $\G$. We construct a $\leq$-universal graph $\Gamma^*$ in $\G$ as follows. Let $V(\Gamma^*):=V(\Gamma)$ and
$$E(\Gamma^*):=\{vw : \text{there are infinitely many independent }v\text{-}w\text{ paths in }\Gamma\}.$$
We need to show that $\Gamma^*\in\G$ and that $G\leq\Gamma^*$ for every graph $G\in\G$.

First, suppose that $\Gamma^*$ is not a member of $\G$ and thus there is a graph $X\in\X$ and a topological embedding $\gamma$ of $X$ in $\Gamma^*$. Then $\gamma$ is also a topological embedding of $X$ in $\Gamma$, contradicting that $\Gamma\in\G$: We recursively find for every edge $vw\in E(X)$ a $\gamma(v)$--$\gamma(w)$ path $P^{vw}$ in $\Gamma$ whose inner vertices avoid $V(X)$ and all paths found in earlier steps. This can be done because there are infinitely many independent $\gamma(v)$--$\gamma(w)$ paths in $\Gamma$ and the inner vertices of $P^{vw}$ only need to avoid finitely many vertices.

Next, we prove that $G\leq\Gamma^*$ for every graph $G\in\G$. Let $G'$ be the graph obtained from $G$ by replacing every edge $e$ with infinitely many independent paths of length 2, whose inner vertices we call $v_0^e,v_1^e,v_2^e,\ldots$. We start by showing that $G'\in\G$. Suppose that $G'$ has a subgraph that is isomorphic to a subdivision $X'$ of a graph $X\in\X$, without loss of generality we assume that $X'\subseteq G'$. By suppressing every vertex of the form $v_i^e$ in $X'$ (reobtaining the edge $e$), we obtain the graph $X''\subseteq G$. Also $X''$ is a subdivision of $X$, since the minimum degree of $X$ is at least 3 and hence we only suppressed subdividing vertices. This contradicts $G\in\G$ and thus we have $G'\in\G$.

Since $\Gamma$ is $\unlhd$-universal in $\G$, there is a topological embedding $\gamma$ of $G'$ in $\Gamma$. Then $\gamma|_{V(G)}$ is an embedding of $G$ in $\Gamma^*$. Indeed, for any edge $vw\in G$ there are infinitely many independent $v$--$w$ paths in $G'$. Therefore, there are infinitely many independent $\gamma(v)$--$\gamma(w)$ paths in $\Gamma$ and thus $\gamma(v)\gamma(w)\in E(\Gamma^*)$.
\end{proof}

The statement of Theorem \ref{excludingtopminors} remains true when we exclude minors instead of topological minors:

\begin{corollary}\label{excludingminors}
If $\X$ is a class of finite graphs with minimum degree at least 3, then $\Forb(\X;\preccurlyeq)$ contains a $\unlhd$-universal graph if and only if it contains a $\leq$-universal graph.
\end{corollary}

\begin{proof}
Let $\Y$ be the class of all models of graphs in $\X$ that are finite and have minimum degree at least 3. We will show that $\Forb(\X;\preccurlyeq)=\Forb(\Y;\unlhd)$, then the claim follows from Theorem \ref{excludingtopminors}. $\Forb(\X;\preccurlyeq)\subseteq\Forb(\Y;\unlhd)$ is clear. For the proof of the converse inclusion, let $G$ be a graph that is not in $\Forb(\X;\preccurlyeq)$. Thus there is a subgraph $X'$ of $G$ which is a model of some $X \in \X$. We choose $X'$ so that it has no vertices of degree 1, the branch sets induce finite trees in $X'$, and there is at most one edge connecting two distinct branch sets. By suppressing all vertices of degree 2 in $X'$, we obtain a topological minor $Y$ of $X'$ with minimum degree at least 3. Since also the minimum degree of $X$ is at least 3, every branch set in $X'$ contains a vertex of degree at least 3 which was not suppressed in the construction of $Y$. Therefore, $X$ is a minor of $Y$. Additionally, $Y$ contains no loops or double edges because every cycle in $X'$ contains at least 3 vertices of degree at least 3 in $X'$, one for each branch set it traverses. Thus, $Y\in\Y$. Since $Y\unlhd X'\subseteq G$, it follows that $G$ is not contained in $\Forb(\Y;\unlhd)$.
\end{proof}

\begin{corollary}
The following classes of graphs do not contain $\unlhd$-universal graphs:
\begin{itemize}
    \item[(1)] $\Forb(K^5,K_{3,3};\unlhd)$ (the planar graphs)
    \item[(2)] $\Forb(K^n;\unlhd)$ for $n\geq 5$
    \item[(3)] $\Forb(K_{n,m};\unlhd)$ for $n,m\geq 3$
    \item[(4)] $\Forb(K^n;\preccurlyeq)$ for $n\geq 5$
\end{itemize}
\end{corollary}

\begin{proof}
The claim follows from Theorem \ref{excludingtopminors}, Corollary \ref{excludingminors} and the fact that none of these classes contain a $\leq$-universal graph: Pach \cite{pach} showed that there is no $\leq$-universal planar graph, Diestel, Halin and Vogler \cite{diestelhalinvogler} showed the same for the classes in (2) and (4) and Diestel \cite{diestel} for the classes in (3).
\end{proof}

\section{Forbidden Subdivided Stars -- Positive Results}\label{section-pos}

In this section we prove Theorem \ref{positive}, which states that there is a $\unlhd$-universal $T$-free graph if $T$ is a subdivided star in which at most 2 of the original edges are subdivided. We use the following notation for subdivided stars:

\begin{definition}\label{subdividedstar}
Let $k\in\N$ and $p_1\leq\ldots\leq p_k\in\N$ and let $P^1,\ldots,P^k$ be disjoint paths such that the length of $P^i$ is $p_i$ for all $i\leq k$.

We write $T(p_1,\ldots,p_k)$ for the graph obtained from the paths $P^1,\ldots,P^k$ by adding a vertex $c$ and identifying one endvertex of each path $P^i$ with $c$. We call $c$ the centre of $T(p_1,\ldots,p_k)$. If $p_1=\ldots=p_k=1$, then we also write $S_k$ for $T(p_1,\ldots,p_k)$.
\end{definition}

\begin{theorem}\label{positive}
Let $k\in\N$ and $p_1\leq\ldots\leq p_k\in\N$. If $k\leq 2$ or if $k\geq 3$ and $p_{k-2}=1$, then there exists a $\unlhd$-universal graph in $\Forb(T(p_1,\ldots,p_k);\leq)$.
\end{theorem}

For the rest of this section, we write $T$ for the graph $T(p_1,\ldots,p_k)$ from Theorem \ref{positive}. For paths $T$, Komjáth, Mekler and Pach \cite{komjathmeklerpach} showed that there exists a $\leq$-universal graph in $\Forb(T;\leq)$, which is also $\unlhd$-universal. In the following, we will therefore assume that $k\geq 3$ and $p_{k-2}=1$.

We begin by proving the existence of a $\unlhd$-universal $S_k$-free graph or, equivalently, a $\unlhd$-universal graph of maximum degree $k-1$. (In Lemma \ref{sk-free}, we prove that there is a $\unlhd$-universal connected $S_k$-free graph. By taking countably many disjoint copies of that universal graph, we obtain a $\unlhd$-universal graph for all $S_k$-free graphs that are not necessarily connected.)

\begin{lemma}\label{sk-free}
There is a $\unlhd_\omega$-universal $\omega$-graph $\Gamma^*$ in the class $\G$ of all connected $S_k$-free $\omega$-graphs on at least 3 vertices, such that $\Gamma^*$ satisfies the following stronger assertion: For every $\omega$-graph $G\in\G$ there is a topological $\omega$-embedding $\gamma^*$ of $G$ in $\Gamma^*$ which is degree preserving, that is $d_G(v)=d_{\Gamma^*}(\gamma^*(v))$ for all $v\in V(G)$.
\end{lemma}

\begin{proof}
Let $[\N]^{<k}$ be the (countable) set of all subsets of $\N$ of size less than $k$ and choose a function $f:\N_{\geq 1}\to[\N]^{<k}$ such that for every $A\in[\N]^{<k}$ there are infinitely many $n\in\N_{\geq 1}$ with $f(n)=A$. Furthermore, let $(R_i:i\in\N)$ be a family of pairwise disjoint rays.
Now we construct $\Gamma^*$ from the graph $\bigcup_{i\in\N}R_i$ by adding for all $j\in\N_{\geq 1}$ a vertex $v_j$, which we join to the $j$th vertex of each ray $R_i$ with $i\in f(j)$ (see Figure \ref{fig:sk-free}).
We define an $\omega$-colouring of $\Gamma^*$ in such a way that for every set $A\in [\N]^{<k}$ and for every colour $c\in\N$, there is an integer $j\in f^{-1}(A)$ such that $v_j$ has colour $c$. This is possible because $f^{-1}(A)$ is infinite for all $A\in [\N]^{<k}$. The vertices of $\bigcup_{i\in\N}R_i$ may be coloured arbitrarily.

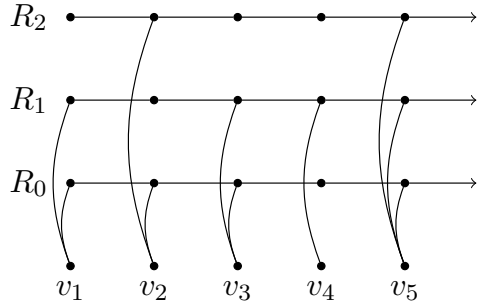
\begin{figure}[h]
\centering
\scalebox{1.1}{
\begin{tikzpicture}

\tikzstyle{v}=[draw,circle,color=black,fill=black,minimum size=2.5pt,inner sep=0pt]

\node[v] (v00) at (0,0) {};
\node[v] (v10) at (1,0) {};
\node[v] (v20) at (2,0) {};
\node[v] (v30) at (3,0) {};
\node[v] (v40) at (4,0) {};
\node[v] (v01) at (0,1) {};
\node[v] (v11) at (1,1) {};
\node[v] (v21) at (2,1) {};
\node[v] (v31) at (3,1) {};
\node[v] (v41) at (4,1) {};
\node[v] (v02) at (0,2) {};
\node[v] (v12) at (1,2) {};
\node[v] (v22) at (2,2) {};
\node[v] (v32) at (3,2) {};
\node[v] (v42) at (4,2) {};
\node[v] (v03) at (0,3) {};
\node[v] (v13) at (1,3) {};
\node[v] (v23) at (2,3) {};
\node[v] (v33) at (3,3) {};
\node[v] (v43) at (4,3) {};
\node (i1) at (5,1) {};
\node (i2) at (5,2) {};
\node (i3) at (5,3) {};

\draw (v01) to (v11);
\draw (v11) to (v21);
\draw (v21) to (v31);
\draw (v31) to (v41);
\draw[->] (v41) to (i1);
\draw (v02) to (v12);
\draw (v12) to (v22);
\draw (v22) to (v32);
\draw (v32) to (v42);
\draw[->] (v42) to (i2);
\draw (v03) to (v13);
\draw (v13) to (v23);
\draw (v23) to (v33);
\draw (v33) to (v43);
\draw[->] (v43) to (i3);

\draw[bend left = 20] (v00) to (v01);
\draw[bend left = 20] (v00) to (v02);
\draw[bend left = 20] (v10) to (v11);
\draw[bend left = 20] (v10) to (v13);
\draw[bend left = 20] (v20) to (v21);
\draw[bend left = 20] (v20) to (v22);
\draw[bend left = 20] (v30) to (v32);
\draw[bend left = 20] (v40) to (v41);
\draw[bend left = 20] (v40) to (v42);
\draw[bend left = 20] (v40) to (v43);

\node at (-0.5,1) {\small{$R_0$}};
\node at (-0.5,2) {\small{$R_1$}};
\node at (-0.5,3) {\small{$R_2$}};
\node at (0,-0.3) {\small{$v_1$}};
\node at (1,-0.3) {\small{$v_2$}};
\node at (2,-0.3) {\small{$v_3$}};
\node at (3,-0.3) {\small{$v_4$}};
\node at (4,-0.3) {\small{$v_5$}};

\end{tikzpicture}
}
\caption{The graph $\Gamma^*$ when $f(1)=\{0,1\}$, $f(2)=\{0,2\}$, $f(3)=\{0,1\}$, $\dots$ drawn without regarding its vertex colouring.}
\label{fig:sk-free}
\end{figure}

$\Gamma^*$ does not contain a copy of $S_k$ because the vertices $v_j$ for $j\in\N$ have degree less than $k$ and all other vertices have degree at most 3. It remains to find a degree preserving topological $\omega$-embedding $\gamma^*$ of $G$ in $\Gamma^*$ for every $\omega$-graph $G\in\G$. We enumerate $E(G)=:\{e_0,e_1,\ldots\}$. For every vertex $v\in V(G)$, we pick an integer
$$j\in f^{-1}(\{i\in\N: v\text{ is incident to }e_i \text{ in } G\})$$
such that $v$ has the same colour in $G$ as $v_j =: \gamma^*(v)$ in $\Gamma$.
Then clearly, $\gamma^*$ is degree preserving. Next, $\gamma^*$ is injective since in a connected graph on at least 3 vertices, the set of incident edges is different for every vertex. Finally, for all edges $e_i=vw\in E(G)$ we have to find internally disjoint $\gamma^*(v)$--$\gamma^*(w)$ paths $P^{e_i}$ in $\Gamma^*$ which avoid the image of $\gamma^*$. Each path $P^{e_i}$ can be built using the ray $R_i$ together with the $\gamma^*(v)$--$R_i$ edge and the $\gamma^*(w)$--$R_i$ edge in $\Gamma^*$.
\end{proof}

We remark that we can also use the construction of $\Gamma$ in the proof of Lemma \ref{sk-free} for building a $\unlhd$-universal graph in the class of all locally finite graphs if we replace $[\N]^{<k}$ by $[\N]^{<\omega}$. A $\leq$-universal locally finite graph however does not exist; this is an easy observation which was first made by de Brujn, see \cite{rado}.

Our next aim in the proof of Theorem \ref{positive} is Lemma \ref{decomposition}, a result on the structure of $T$-free graphs. In the proof of Lemma \ref{decomposition}, we analyse the block structure of $T$-free graphs (a block of a graph $G$ is a maximal connected subgraph of $G$ without a cutvertex). We need the following lemmas:

\begin{lemma}[\cite{diestel} Chapter 1, Exercise 3]\label{cycle}
Let $n\in\N$. Every 2-connected graph $G$ containing a path of length $n^2$ contains a cycle of length at least $n$.
\end{lemma}

We recall that $k\geq 3$, $p_{k-2}=1$ and $T=T(p_1,\ldots,p_k)$.

\begin{lemma}\label{boundeddiam}
Let $G$ be a $T$-free graph and let $B$ be a block of $G$ that contains a vertex $v\in V(B)$ with $d_G(v)\geq k$. Then there is no path of length $m:=(k+1)^2(2p_k)^2$ in $B$.
\end{lemma}

\begin{proof}
Suppose for a contradiction that $B$ contains a path of length $m$. Therefore, $|B|\geq 3$ and it follows that $B$ is 2-connected. By Lemma \ref{cycle}, there is a cycle $C$ in $B$ of length at least $\sqrt{m}=(k+1)\cdot 2p_k$.
First we consider the case that $S_B(v)\cap C=\emptyset$. By 2-connectedness of $B$ there are two disjoint $S_B(v)$--$C$ paths $P^1$ and $P^2$ in $B$. Denote the endvertex of $P^i$ in $C$ by $x_i$. We can find an $x_1$--$x_2$ path $Q$ in $C$ of length at least $2p_k+1$ since $|C|\geq 8p\geq 2(2p_k+1)$. However, there is a copy of $T$ in $S_G(v)\cup P^1\cup P^2\cup Q$, a contradiction.

If $S_B(v)\cap C=\{w\}$ for some vertex $w$, then let $P^1=w$ be a path of length 1. Note that $w\neq v$. Since $B$ is 2-connected, there is an $S_B(v)$--$C$ path $P^2$ in $B-w$. Now we can prove that $T\leq G$ as in the case $S_B(v)\cap C=\emptyset$.

Finally, suppose that $|V(S_B(v))\cap V(C)|\geq 2$. There are at most $|S_B(v)|=k+1$ many distinct $S_B(v)$-paths contained in $C$. Therefore, there must be an $S_B(v)$-path $P$ in $C$ of length at least $2p_k$ since $|C|\geq(k+1)\cdot 2p_k$. This is a contradiction as $S_G(v)\cup P$ contains a copy of $T$.
\end{proof}

\begin{lemma}\label{decomposition}
Let $m$ be as in Lemma \ref{boundeddiam} and let $G$ be a connected $T$-free graph with $P_{4p_k m}\leq G$. Then there is a connected induced subgraph $G^*$ of $G$ and for every vertex $v\in V(G^*)$ there is a connected induced subgraph $G_v$ of $G$ such that the following properties hold:

\begin{itemize}
\item[(1)] $G=G^*\cup\bigcup_{v\in V(G^*)}G_v$,
\item[(2)] $G^*\cap G_v=\{v\}$ and $G_v\cap G_w=\emptyset$ for all $v\neq w\in V(G^*)$,
\item[(3)] $d_G(v)<k$ for all $v\in V(G^*)$,
\item[(4)] $P_{2p_k}\leq G^*$, and
\item[(5)] $P_{8p_k+2m}\not\leq G_v$ for all $v\in V(G^*)$.
\end{itemize}
\end{lemma}

\begin{proof}
Let $A$ be the set of cutvertices in $G$ and $\B$ the set of blocks of $G$ and $K$ the block graph of $G$. Recall that $V(K)=A\cup\B$ and $E(G)=\{aB:a\in A,B\in\B, a\in B\}$ and that $K$ is a tree. Since $G$ contains a path of length $4p_k m$, there is either a block $B\in\B$ containing a path of length $m$ or there is a path $P=B_1a_1B_2a_2\ldots B_{4p_k-1}a_{4p_k-1}B_{4p_k}$ in $K$ containing $4p_k$ blocks. In the first case let $H^*:=B$ and in the second case let $$H^*:=\bigcup_{p_k+1\leq i\leq 3p_k}B_i.$$

In both cases, there is clearly a path of length $2p_k$ in $H^*$. We also have $d_G(v)<k$ for all $v\in V(H^*)$: In the first case, this holds by Lemma \ref{boundeddiam}. For the second case, let us suppose that there are a block $B_i$ with $p_k+1\leq i\leq 3p_k$ and a vertex $v\in B_i$ with $d_G(v)\geq k$.
Since $B_i$ is 2-connected, there is an $S_{B_i}(v)$--$a_{i-1}$ path $P^1$ and an $S_{B_i}(v)$--$a_i$ path $P^2$ in $B_i$ such that $P^1$ and $P^2$ are disjoint. Now we extend $S_G(v)\cup P^1\cup P^2$ to a copy of $T$ in $G$ by extending the path $P^1$ into the blocks $B_j$ with $j<i$ and extending $P^2$ into the blocks $B_j$ with $j>i$. This contradicts $G$ being $T$-free. Hence $d_G(v)<k$ for all $v\in V(H^*)$.

If $K'$ is a subtree of $K$, we write $\bigcup K'$ for $\bigcup\{B\in\B:B\in V(K')\}$. Let $\Tilde{K}^*$ be the maximal subtree of $K$ that contains all vertices of $K$ which are blocks of $H^*$, such that there is no vertex $v\in\bigcup\Tilde{K}^*$ with $d_G(v)\geq k$. We delete all leaves of $\Tilde{K}^*$ that are cutvertices in $G$ and call the resulting graph $K^*$. Then
$$G^*:=\bigcup K^*$$
satisfies property (3) by definition of $K^*$ and property (4) since already $H^*$ contains a path of length $2p_k$.

For all components $H$ of $K-K^*$ we have $G^*\cap\bigcup H=\{v\}$ for some vertex $v\in V(G^*)\cap A$. Let $K_v:=H$ and $G_v:=\bigcup K_v$. For the proof of property (5), we begin by decomposing $G_v$ into a family $\G_v$ of graphs that only intersect in $v$. Let $\K_v$ be the set of components of $K_v-v$ and $\G_v:=\{\bigcup K':K'\in\K_v\}$. For all $K'\in\K_v$ we show that the graph $G':=\bigcup K'\in\G_v$ (Figure \ref{blockstructure}) is $P_{4p_k+m}$-free, which implies that $G_v$ is $P_{8p_k+2m}$-free and thus proves property (5).

Let $D$ be the block of $G'$ containing $v$. By maximality of $K^*$, we know that $D$ contains a vertex $w$ with $d_G(w)\geq k$. Therefore, Lemma \ref{boundeddiam} implies that $D$ does not contain a path of length $m$. Additionally, for every component $K''$ of $K'-D$ the graph $G'':=\bigcup K''$ does not contain a path of length $2p_k$ (Figure~\ref{blockstructure}): Suppose that $G''$ contains a path $P''$ of length $2p_k$. We denote the unique vertex in $D\cap G''$ by $u$.
Let $P_x$ be an $S_D(w)$--$x$ path in $D$ for $x\in\{v,u\}$ such that $P_v$ and $P_u$ are disjoint.
Further, let $P^*$ be a path of length $2p_k$ in $G^*$, which exists by property (4), and let $Q_v$ be a $v$--$P^*$ path in $G^*$ and $Q_u$ a $u$--$P''$ path in $G''$. However,
$$S_G(w)\cup P_v\cup P_u\cup Q_v\cup Q_u\cup P^*\cup P''\subseteq G$$
contains a copy of $T$. Therefore, $G'' = \bigcup K''$ does not contain a path of length $2p_k$ for every choice of $K''$. Together with $D$ being $P_m$-free, this implies that there is no path of length $4p_k+m$ in $G'$.

Finally, for all $v\in V(G^*)$ for which $G_v$ has not been defined yet, let $G_v:=\{v\}$. Then properties (1) and (2) are clearly satisfied.
\end{proof}

\begin{figure}[ht]
\centering
\includegraphics[scale=0.75]{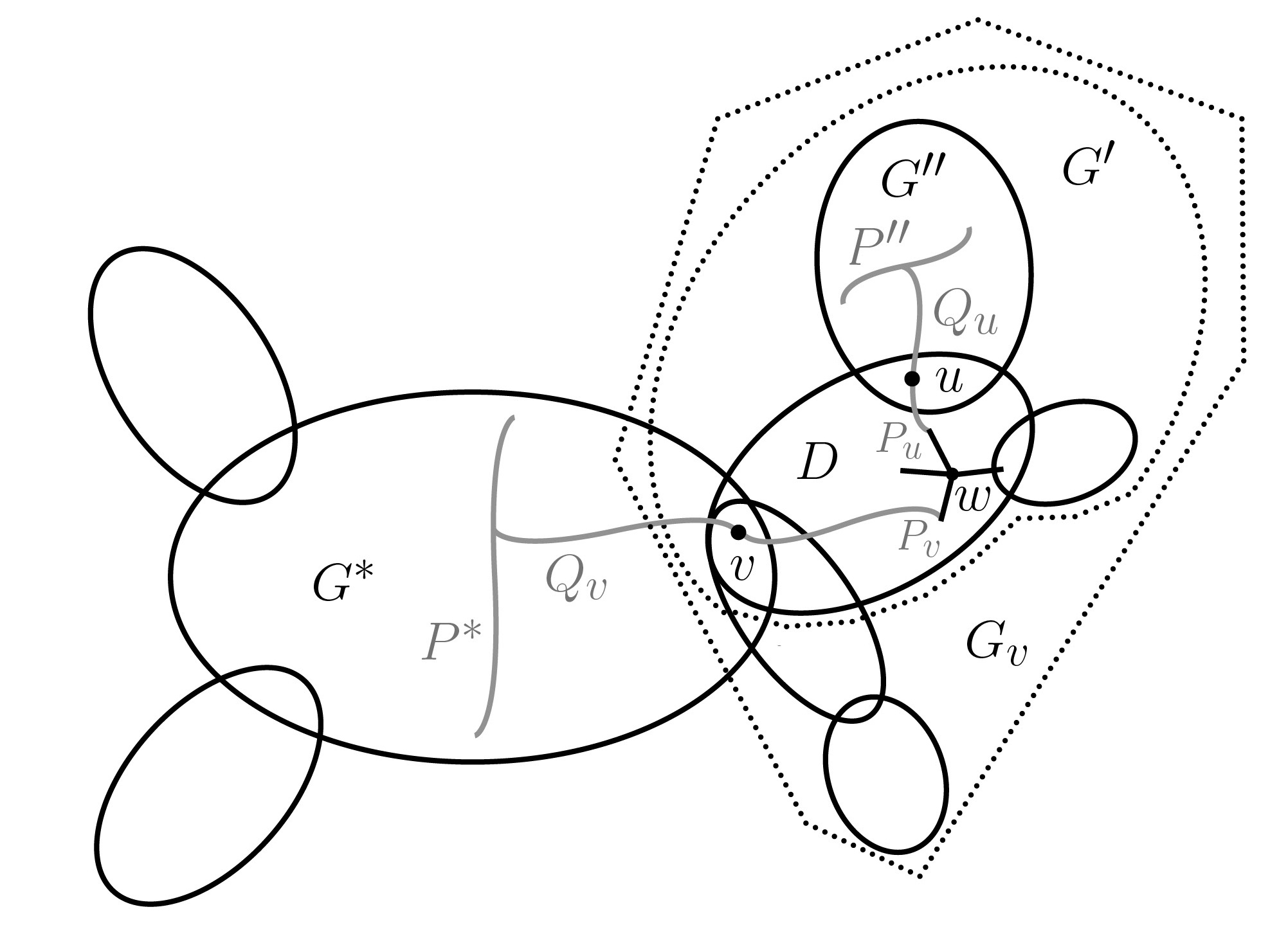}
\caption{The graph $G$ and its subgraphs}
\label{blockstructure}
\end{figure}

For the proof of Theorem \ref{positive}, we need the following Lemma by Cherlin and Tallgren:

\begin{lemma}[\cite{cherlintallgren}]\label{pn-free}
Let $\alpha\in\N$ and let $\X$ be a finite set of finite connected $\alpha$-graphs. Suppose that for some $n\in\N$, every $\alpha$-coloured version of $P_n$ is contained in $\X$. Then there is a $\leq_\alpha$-universal $\alpha$-graph in $\Forb(\X;\leq_\alpha)$.
\end{lemma}

The idea for our proof of Theorem \ref{positive} is to build a $\unlhd$-universal graph $\Gamma$ in $\Forb(T;\leq)$ by starting with the $\unlhd$-universal $S_k$-free graph $\Gamma^*$ from Lemma \ref{sk-free} and attaching to every vertex $v\in V(\Gamma^*)$ a $\leq$-universal $P_{8p_k+2m}$-free graph $\Gamma_v$ that exists by Lemma \ref{pn-free}. If $G$ is any connected $T$-free graph, we want to decompose $G$ as in Lemma \ref{decomposition} and find a topological embedding $\gamma$ of $G$ in $\Gamma$ by embedding $G^*$ in $\Gamma^*$ and each graph $G_v$ in $\Gamma_{\gamma(v)}$.

\begin{proof}[Proof of Theorem \ref{positive}]
Recall that we have already seen that the theorem holds for $k = 2$, so we can assume that $k \geq 3$.
We will build a graph $\Gamma$ that is $\unlhd$-universal in the class of all connected $T$-free graphs containing $P_{4p_k m}$. By Lemma \ref{pn-free} (applied with $\alpha=1$), there exists a $\leq$-universal graph $\Gamma'$ in $\Forb(T,P_{4p_k m};\leq)$. Then the disjoint union of $\Gamma'$ with countably many copies of $\Gamma$ is $\unlhd$-universal in $\Forb(T;\leq)$.

If $H$ is a 2-graph with exactly one vertex $v$ of colour 1, then we write $\overline{H}$ for the graph obtained from $H$ by attaching a path of length $p_k$ to $v$. Let
\begin{itemize}
\item $\X_1$ be the set of all connected 2-graphs $H$ with $V(H)\subseteq\{1,\ldots,|T|\}$ such that $H$ contains exactly one vertex of colour 1 and $T\leq\overline{H}$,
\item $\X_2$ be the set of all 2-coloured versions of the graph $P_{8p_k+2m}$, and
\item $\X_3$ be the set of all 2-coloured versions $P$ of the paths $P_i$ for $i<8p_k+2m$ such that all the inner vertices of $P$ have colour 0 and the endvertices have colour 1.
\end{itemize}
For all $n<k$, let
\begin{itemize}
\item $\X_4^n$ be the one-element set containing the 2-graph $S_{n+1}$ so that the centre of $S_{n+1}$ has colour 1 and all other vertices have colour 0, and
\item $\X^n:=\X_1\cup \X_2\cup \X_3\cup \X_4^n$.
\end{itemize}
Since $\X_2\subseteq\X^n$, there exists a $\leq_2$-universal $\X^n$-free 2-graph $\Gamma^n$ by Lemma \ref{pn-free}. We fix an enumeration $\Gamma^n_0,\Gamma^n_1,\Gamma^n_2\ldots$ of all components of $\Gamma^n$ which contain a vertex of degree at least $n$ that has colour 1. Note that each component $\Gamma^n_i$ contains exactly one vertex $u$ of colour 1 because $\Gamma^n_i$ is connected and $\X_2\cup\X_3$-free. The degree of $u$ must be $n$ since $\Gamma^n_i$ is $\X_4^n$-free.

Let $\Gamma^*$ be the $\omega$-graph from Lemma \ref{sk-free}. For every vertex $v\in V(\Gamma^*$), let $\Gamma_v:=\Gamma^{n(v)}_{c(v)}$ where
$$n(v):=k-d_{\Gamma^*}(v)-1$$
and $c(v)$ is the colour of $v$ in $\Gamma^*$. Notice that $n(v)\geq 0$ since $\Gamma^*$ is $S_k$-free. By renaming vertices of each $\Gamma_v$, we obtain for all $v\neq w\in V(\Gamma^*)$ that $\Gamma_v$ and $\Gamma_w$ are disjoint, the unique vertex of colour 1 in $\Gamma_v$ is $v$, and $\Gamma^*\cap \Gamma_v=\{v\}$. We define (the uncoloured graph)
$$\Gamma:=\Gamma^*\cup\bigcup_{v\in V(\Gamma^*)}\Gamma_v.$$
We need to show that $\Gamma$ is $T$-free and that $G\unlhd\Gamma$ for every connected $T$-free graph $G$ with $P_{4p_k m}\leq G$.

Suppose for a contradiction that there is an embedding $\gamma$ of $T$ in $\Gamma$ and let $z$ be the centre of $T$. We have $\gamma(z)\in \Gamma_v-\Gamma^*$ for some $v\in V(\Gamma^*)$ since any vertex $u\in V(\Gamma^*)$ satisfies $d_\Gamma(u)=d_{\Gamma^*}(u)+d_{\Gamma_u}(u)=k-1$. Therefore, $\gamma(T)\cap \Gamma^*$ is either empty or a path of length at most $p_k$ with endvertex $v$, which contradicts $\Gamma_v$ being $\X_1$-free.

Now let $G$ be an arbitrary connected $T$-free graph with $P_{4p_k m}\leq G$ and consider the graph $G^*$ and the graphs $G_v$ from Lemma \ref{decomposition}.
We furnish each $G_v$ with a 2-colouring by colouring $v$ with 1 and all other vertices with 0. Then $G_v$ is $\X_1$-free: Suppose that there is a graph $H\in\X_1$ and a 2-embedding $f$ from $H$ to $G_v$. Since $G^*$ is connected with $P_{2p_k}\leq G^*$ by Lemma \ref{decomposition} (4), there is a path $P$ of length $p_k$ in $G^*$ with endvertex $v$. This is a contradiction because $G[f(H)]\cup P\subseteq G$ contains a copy of $T$. Hence $G_v$ is $\X_1$-free. $G_v$ is also $\X_2$-free by Lemma \ref{decomposition} (5) and $\X_3$-free since $v$ is the only vertex of colour 1 in $G_v$. Lastly, we have $d_G(v)<k$ by Lemma \ref{decomposition} (3) and therefore
$$d_{G_v}(v)=d_G(v)-d_{G^*}(v)\leq k-d_{G^*}(v)-1=:n'(v),$$
which shows that $G_v$ is $\X_4^{n'(v)}$-free. Therefore, there is a 2-embedding of $G_v$ in $\Gamma^{n'(v)}$, and since $G_v$ is connected, there is an integer $c'(v)\in\N$ such that the image of this 2-embedding is contained in $\Gamma^{n'(v)}_{c'(v)}$.

Next, define an $\omega$-colouring of $G^*$ by giving each vertex $v\in V(G^*)$ the colour $c'(v)$. Since $G^*$ is $S_k$-free by Lemma \ref{decomposition} (3), Lemma \ref{sk-free} implies that there is a topological $\omega$-embedding $\gamma^*$ of $G^*$ in $\Gamma^*$ with $d_{G^*}(v)=d_{\Gamma^*}(\gamma(v))$ and hence $n'(v)=n(\gamma^*(v))$ for all $v\in V(G^*)$. For every vertex $v\in V(G^*)$, we have $c'(v)=c(\gamma^*(v))$ because $\gamma$ respects colouring. This means that $\Gamma_{\gamma^*(v)}$ is a copy of $\Gamma_{c'(v)}^{n'(v)}$.
Therefore, there exists a 2-embedding $\gamma_v$ of $G_v$ in $\Gamma_v$ and we have $\gamma_v(v)=\gamma^*(v)$ since $\gamma_v$ respects colouring. Thus
$$\gamma^*\cup\bigcup_{v\in V(G^*)}\gamma_v$$
is a topological embedding of $G$ in $\Gamma$.
\end{proof}

\section{Forbidden Subdivided Stars -- Negative Results}\label{section-neg}

In this section, we prove that there is no $\unlhd$-universal $T$-free graph if $T$ is a subdivided star in which at least 3 of the original edges are subdivided.
The proof relies on ideas from Cherlin and Shelah in \cite[Proposition 3.3]{cherlinshelah}.

\begin{theorem}\label{negative}
Let $k\in\N$ and $p_1\leq\ldots\leq p_k\in\N$. If $k\geq 3$ and $p_{k-2}\geq 2$, then there exists no $\unlhd$-universal graph in $\Forb(T(p_1,\ldots,p_k);\leq)$.
\end{theorem}


\begin{proof}
Let $T:=T(p_1,\ldots,p_k)$ and let $P^1,\ldots,P^k$ be the corresponding paths from Definition \ref{subdividedstar}. Let $m$ be minimal with $p_m>1$ and define $n:=k-m+1$. Then exactly $n$ of the paths $P^1,\ldots,P^k$ have length at least $p_m$ and the other paths have length 1.

Let $H_1$ be the graph consisting of two vertices $x_1$ and $x_2$ together with infinitely many independent $x_1$--$x_2$ paths of length $p_m$ and let $H_2$ be the graph with vertex set $\{x_1,x_2,y_1,y_2,\ldots,y_{p_m}\}$ where every two vertices are adjacent except for $x_1$ and $x_2$.

We use $H_1$ and $H_2$ to construct an uncountable family of $T$-free graphs. Consider the uncountable set $A\subseteq \{1,2\}^\N$ of all sequences that alternatingly contain one or two 1's and one 2, beginning with a 1. For all $\alpha\in A$, let $U_\alpha$ be an $(n-1)$-regular tree with root $r_\alpha$. For each edge $vw\in E(U_\alpha)$, we denote the minimum of $d_{U_\alpha}(r_\alpha,v)$ and $d_{U_\alpha}(r_\alpha,w)$ by $\ell(vw)$. For all $\alpha\in A$, we define a graph $G_\alpha$ as follows: We begin with $U_\alpha$, replace every edge $vw\in E(U_\alpha)$ with a copy of $H_{\alpha(\ell(vw))}$, and identify $x_1$ with $v$ and $x_2$ with $w$ (see Figure \ref{fig:G_alpha}). Note that for every vertex $v\in V(U_\alpha)$ there is at least one edge $e\in E(U_\alpha)$ which was replaced by a copy of $H_1$. Hence $v$ has infinite degree in $G_\alpha$.

\begin{figure}[h]
\centering
\scalebox{1.1}{
\begin{tikzpicture}
    \tikzstyle{v}=[draw,circle,fill=black,minimum size=4.5pt,inner sep=0pt]
    \tikzstyle{u}=[draw,circle,fill=black,minimum size=2.8pt,inner sep=0pt]
    \tikzstyle{fat} = [draw,line width=2pt,black!100]
    \tikzstyle{thin} = [draw,line width=0.5pt,black!60]
    
    \node[scale=1.2] at (5.5,3.55) {. . .};
    
    \node[v] (d0) at (0,3) {};
    \node[v] (d1) at (1,3) {};
    \node[v] (d2) at (2,3) {};
    \node[v] (d3) at (3,3) {};
    \node[v] (d4) at (4,3) {};
    \node[v] (d5) at (5,3) {};
    \node[v] (d6) at (6,3) {};
    \node[v] (d7) at (7,3) {};
    \node[v] (d8) at (8,3) {};
    \node[v] (d9) at (9,3) {};
    \node[v] (d10) at (10,3) {};
    \node[v] (d11) at (11,3) {};
    
    \node[v] (c0) at (0.5,2) {};
    \node[v] (c1) at (2.5,2) {};
    \node[v] (c2) at (4.5,2) {};
    \node[v] (c3) at (6.5,2) {};
    \node[v] (c4) at (8.5,2) {};
    \node[v] (c5) at (10.5,2) {};
    
    \node[v] (b0) at (1.5,1) {};
    \node[v] (b1) at (5.5,1) {};
    \node[v] (b2) at (9.5,1) {};
    
    \node[v] (a0) at (5.5,0) {};

    \node[u] (x0) at (3.4,0.1) {};
    \node[u] (x1) at (3.44,0.26) {};
    \node[u] (x2) at (3.465,0.36) {};
    \node[u] (x3) at (3.485,0.44) {};
    \node[u] (x4) at (3.495,0.48) {};
    \node[u] (x5) at (3.5,0.5) {};
    \draw[fat] (a0) -- (x0);
    \draw[thin] (a0) -- (x1);
    \draw[thin] (a0) -- (x2);
    \draw[thin] (a0) -- (x3);
    \draw[thin] (a0) -- (x4);
    \draw[thin] (a0) -- (x5);
    \draw[fat] (b0) -- (x0);
    \draw[thin] (b0) -- (x1);
    \draw[thin] (b0) -- (x2);
    \draw[thin] (b0) -- (x3);
    \draw[thin] (b0) -- (x4);
    \draw[thin] (b0) -- (x5);

    \node[u] (x6) at (4.9,0.5) {};
    \node[u] (x7) at (5.14,0.5) {};
    \node[u] (x8) at (5.29,0.5) {};
    \node[u] (x9) at (5.41,0.5) {};
    \node[u] (x10) at (5.47,0.5) {};
    \node[u] (x11) at (5.5,0.5) {};
    \draw[fat] (a0) -- (x6) {};
    \draw[thin] (a0) -- (x7) {};
    \draw[thin] (a0) -- (x8) {};
    \draw[thin] (a0) -- (x9) {};
    \draw[thin] (a0) -- (x10) {};
    \draw[thin] (a0) -- (x11) {};
    \draw[fat] (b1) -- (x6) {};
    \draw[thin] (b1) -- (x7) {};
    \draw[thin] (b1) -- (x8) {};
    \draw[thin] (b1) -- (x9) {};
    \draw[thin] (b1) -- (x10) {};
    \draw[thin] (b1) -- (x11) {};

    \node[u] (x12) at (7.6,0.1) {};
    \node[u] (x13) at (11 - 3.44,0.26) {};
    \node[u] (x14) at (11 - 3.465,0.36) {};
    \node[u] (x15) at (11 - 3.485,0.44) {};
    \node[u] (x16) at (11 - 3.495,0.48) {};
    \node[u] (x17) at (11 - 3.5,0.5) {};
    \draw[fat] (a0) -- (x12);
    \draw[thin] (a0) -- (x13);
    \draw[thin] (a0) -- (x14);
    \draw[thin] (a0) -- (x15);
    \draw[thin] (a0) -- (x16);
    \draw[thin] (a0) -- (x17);
    \draw[fat] (b2) -- (x12);
    \draw[thin] (b2) -- (x13);
    \draw[thin] (b2) -- (x14);
    \draw[thin] (b2) -- (x15);
    \draw[thin] (b2) -- (x16);
    \draw[thin] (b2) -- (x17);
    
    \node[u] (y0) at (0.6,1.1) {};
    \node[u] (y1) at (0.76,1.26) {};
    \node[u] (y2) at (0.86,1.36) {};
    \node[u] (y3) at (0.94,1.44) {};
    \node[u] (y4) at (0.98,1.48) {};
    \node[u] (y5) at (1,1.5) {};
    \draw[fat] (b0) -- (y0);
    \draw[thin] (b0) -- (y1);
    \draw[thin] (b0) -- (y2);
    \draw[thin] (b0) -- (y3);
    \draw[thin] (b0) -- (y4);
    \draw[thin] (b0) -- (y5);
    \draw[fat] (c0) -- (y0);
    \draw[thin] (c0) -- (y1);
    \draw[thin] (c0) -- (y2);
    \draw[thin] (c0) -- (y3);
    \draw[thin] (c0) -- (y4);
    \draw[thin] (c0) -- (y5);

    \node[u] (y6) at (2.4,1.1) {};
    \node[u] (y7) at (2.24,1.26) {};
    \node[u] (y8) at (2.14,1.36) {};
    \node[u] (y9) at (2.06,1.44) {};
    \node[u] (y10) at (2.02,1.48) {};
    \node[u] (y11) at (2,1.5) {};
    \draw[fat] (b0) -- (y6);
    \draw[thin] (b0) -- (y7);
    \draw[thin] (b0) -- (y8);
    \draw[thin] (b0) -- (y9);
    \draw[thin] (b0) -- (y10);
    \draw[thin] (b0) -- (y11);
    \draw[fat] (c1) -- (y6);
    \draw[thin] (c1) -- (y7);
    \draw[thin] (c1) -- (y8);
    \draw[thin] (c1) -- (y9);
    \draw[thin] (c1) -- (y10);
    \draw[thin] (c1) -- (y11);

    \node[u] (y12) at (4 + 0.6,1.1) {};
    \node[u] (y13) at (4 + 0.76,1.26) {};
    \node[u] (y14) at (4 + 0.86,1.36) {};
    \node[u] (y15) at (4 + 0.94,1.44) {};
    \node[u] (y16) at (4 + 0.98,1.48) {};
    \node[u] (y17) at (4 + 1,1.5) {};
    \draw[fat] (b1) -- (y12);
    \draw[thin] (b1) -- (y13);
    \draw[thin] (b1) -- (y14);
    \draw[thin] (b1) -- (y15);
    \draw[thin] (b1) -- (y16);
    \draw[thin] (b1) -- (y17);
    \draw[fat] (c2) -- (y12);
    \draw[thin] (c2) -- (y13);
    \draw[thin] (c2) -- (y14);
    \draw[thin] (c2) -- (y15);
    \draw[thin] (c2) -- (y16);
    \draw[thin] (c2) -- (y17);

    \node[u] (y18) at (4 + 2.4,1.1) {};
    \node[u] (y19) at (4 + 2.24,1.26) {};
    \node[u] (y20) at (4 + 2.14,1.36) {};
    \node[u] (y21) at (4 + 2.06,1.44) {};
    \node[u] (y22) at (4 + 2.02,1.48) {};
    \node[u] (y23) at (4 + 2,1.5) {};
    \draw[fat] (b1) -- (y18);
    \draw[thin] (b1) -- (y19);
    \draw[thin] (b1) -- (y20);
    \draw[thin] (b1) -- (y21);
    \draw[thin] (b1) -- (y22);
    \draw[thin] (b1) -- (y23);
    \draw[fat] (c3) -- (y18);
    \draw[thin] (c3) -- (y19);
    \draw[thin] (c3) -- (y20);
    \draw[thin] (c3) -- (y21);
    \draw[thin] (c3) -- (y22);
    \draw[thin] (c3) -- (y23);

    \node[u] (y24) at (11 - 2.4,1.1) {};
    \node[u] (y25) at (11 - 2.24,1.26) {};
    \node[u] (y26) at (11 - 2.14,1.36) {};
    \node[u] (y27) at (11 - 2.06,1.44) {};
    \node[u] (y28) at (11 - 2.02,1.48) {};
    \node[u] (y29) at (11 - 2,1.5) {};
    \draw[fat] (b2) -- (y24);
    \draw[thin] (b2) -- (y25);
    \draw[thin] (b2) -- (y26);
    \draw[thin] (b2) -- (y27);
    \draw[thin] (b2) -- (y28);
    \draw[thin] (b2) -- (y29);
    \draw[fat] (c4) -- (y24);
    \draw[thin] (c4) -- (y25);
    \draw[thin] (c4) -- (y26);
    \draw[thin] (c4) -- (y27);
    \draw[thin] (c4) -- (y28);
    \draw[thin] (c4) -- (y29);

    \node[u] (y30) at (11 - 0.6,1.1) {};
    \node[u] (y31) at (11 - 0.76,1.26) {};
    \node[u] (y32) at (11 - 0.86,1.36) {};
    \node[u] (y33) at (11 - 0.94,1.44) {};
    \node[u] (y34) at (11 - 0.98,1.48) {};
    \node[u] (y35) at (11 - 1,1.5) {};
    \draw[fat] (b2) -- (y30);
    \draw[thin] (b2) -- (y31);
    \draw[thin] (b2) -- (y32);
    \draw[thin] (b2) -- (y33);
    \draw[thin] (b2) -- (y34);
    \draw[thin] (b2) -- (y35);
    \draw[fat] (c5) -- (y30);
    \draw[thin] (c5) -- (y31);
    \draw[thin] (c5) -- (y32);
    \draw[thin] (c5) -- (y33);
    \draw[thin] (c5) -- (y34);
    \draw[thin] (c5) -- (y35);
    
    \node[u] (z0) at (0.13,2.44) {};
    \node[u] (z1) at (0.37,2.56) {};
    \draw[thin] (z0) -- (z1);
    \draw[fat] (c0) -- (z0);
    \draw[thin] (c0) -- (z1);
    \draw[fat] (d0) -- (z0);
    \draw[thin] (d0) -- (z1);

    \node[u] (z2) at (0.87,2.44) {};
    \node[u] (z3) at (0.63,2.56) {};
    \draw[thin] (z2) -- (z3);
    \draw[fat] (c0) -- (z2);
    \draw[thin] (c0) -- (z3);
    \draw[fat] (d1) -- (z2);
    \draw[thin] (d1) -- (z3);

    \node[u] (z4) at (2 + 0.13,2.44) {};
    \node[u] (z5) at (2 + 0.37,2.56) {};
    \draw[thin] (z4) -- (z5);
    \draw[fat] (c1) -- (z4);
    \draw[thin] (c1) -- (z5);
    \draw[fat] (d2) -- (z4);
    \draw[thin] (d2) -- (z5);

    \node[u] (z6) at (2 + 0.87,2.44) {};
    \node[u] (z7) at (2 + 0.63,2.56) {};
    \draw[thin] (z6) -- (z7);
    \draw[fat] (c1) -- (z6);
    \draw[thin] (c1) -- (z7);
    \draw[fat] (d3) -- (z6);
    \draw[thin] (d3) -- (z7);

    \node[u] (z8) at (4 + 0.13,2.44) {};
    \node[u] (z9) at (4+ 0.37,2.56) {};
    \draw[thin] (z8) -- (z9);
    \draw[fat] (c2) -- (z8);
    \draw[thin] (c2) -- (z9);
    \draw[fat] (d4) -- (z8);
    \draw[thin] (d4) -- (z9);

    \node[u] (z10) at (4 + 0.87,2.44) {};
    \node[u] (z11) at (4 + 0.63,2.56) {};
    \draw[thin] (z10) -- (z11);
    \draw[fat] (c2) -- (z10);
    \draw[thin] (c2) -- (z11);
    \draw[fat] (d5) -- (z10);
    \draw[thin] (d5) -- (z11);

    \node[u] (z12) at (7 - 0.87,2.44) {};
    \node[u] (z13) at (7 - 0.63,2.56) {};
    \draw[thin] (z12) -- (z13);
    \draw[fat] (c3) -- (z12);
    \draw[thin] (c3) -- (z13);
    \draw[fat] (d6) -- (z12);
    \draw[thin] (d6) -- (z13);

    \node[u] (z14) at (7 - 0.13,2.44) {};
    \node[u] (z15) at (7 - 0.37,2.56) {};
    \draw[thin] (z14) -- (z15);
    \draw[fat] (c3) -- (z14);
    \draw[thin] (c3) -- (z15);
    \draw[fat] (d7) -- (z14);
    \draw[thin] (d7) -- (z15);

    \node[u] (z16) at (9 - 0.87,2.44) {};
    \node[u] (z17) at (9 - 0.63,2.56) {};
    \draw[thin] (z16) -- (z17);
    \draw[fat] (c4) -- (z16);
    \draw[thin] (c4) -- (z17);
    \draw[fat] (d8) -- (z16);
    \draw[thin] (d8) -- (z17);

    \node[u] (z18) at (9 - 0.13,2.44) {};
    \node[u] (z19) at (9 - 0.37,2.56) {};
    \draw[thin] (z18) -- (z19);
    \draw[fat] (c4) -- (z18);
    \draw[thin] (c4) -- (z19);
    \draw[fat] (d9) -- (z18);
    \draw[thin] (d9) -- (z19);

    \node[u] (z20) at (11 - 0.87,2.44) {};
    \node[u] (z21) at (11 - 0.63,2.56) {};
    \draw[thin] (z20) -- (z21);
    \draw[fat] (c5) -- (z20);
    \draw[thin] (c5) -- (z21);
    \draw[fat] (d10) -- (z20);
    \draw[thin] (d10) -- (z21);

    \node[u] (z22) at (11 - 0.13,2.44) {};
    \node[u] (z23) at (11 - 0.37,2.56) {};
    \draw[thin] (z22) -- (z23);
    \draw[fat] (c5) -- (z22);
    \draw[thin] (c5) -- (z23);
    \draw[fat] (d11) -- (z22);
    \draw[thin] (d11) -- (z23);
    
\end{tikzpicture}
}
\caption{The graph $G_\alpha$ when $p_m = 2$, $n = 4$ and $\alpha = (1, 1, 2, \dots)$. The fat edges form a subdivision of $U_\alpha$ in $G_\alpha$.}
\label{fig:G_alpha}
\end{figure}
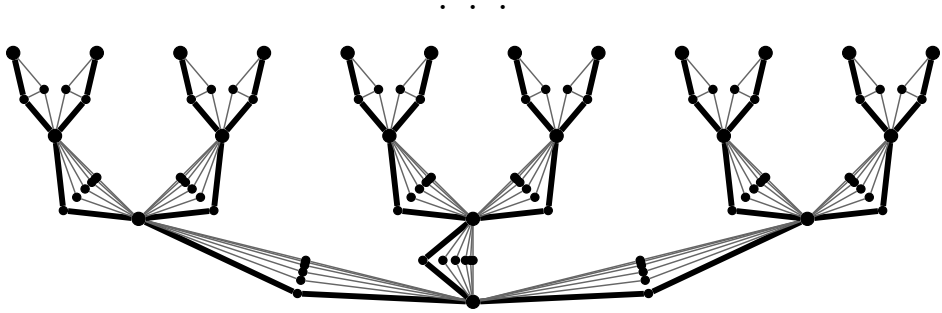

We now prove two crucial properties of $G_\alpha$ for all $\alpha\in A$:

\smallskip
\noindent\textit{Claim 1:} $G_\alpha$ is $T$-free.
\smallskip

\noindent\textit{Proof of Claim 1:}
Suppose that there is a copy of $T$ in $G_\alpha$ and let $v$ be its centre. It is impossible that $v\in V(G_\alpha)\setminus V(U_\alpha)$: If $v$ is contained in a copy of $H_1$, then the degree of $v$ in $G_\alpha$ is 2 and therefore $v$ cannot be the centre of a copy of $T$. Now suppose that $v$ is contained in a copy of $H_2$. The copy of $T$ contains 3 paths of length at least $p_m$ starting in $v$. Each of them has at least $p_m+1$ vertices and must therefore use one of the vertices corresponding to $x_1$ or $x_2$ in the copy of $H_2$ which is a contradiction. Therefore, $v$ must be contained in $V(U_\alpha)$. However, $v$ is only contained in $n-1$ copies of $H_1$ or $H_2$ and each of them can only contain inner vertices of at most one of the paths $P^i$ for $i\geq m$, a contradiction.
\hfill $\Box$ \textit{(Claim 1)}

\smallskip
\noindent\textit{Claim 2:} Every subdivision $G'_\alpha$ of $G_\alpha$ with at least one subdividing vertex $v$ contains a subgraph isomorphic to $T$.
\smallskip

\noindent\textit{Proof of Claim 2:}
Let $H$ be the subdivision of a copy of $H_i$ for $i=1$ or $i=2$ in $G'_\alpha$ which contains $v$ and let $\gamma$ be a topological embedding of $H_i$ in $H$. Then we can find an embedding of $T$ in $G'_\alpha$ as follows. We map the centre of $T$ to $\gamma(x_1)$, and $P^m$ to a path of length $p_m$ in $H$ containing $v$ but not $\gamma(x_2)$, and $P^i$ for $i>m$ to paths each containing one neighbour of $\gamma(x_1)$ in $U_\alpha$. Since $\gamma(x_1)$ is contained in $V(U_\alpha)$, it has infinite degree in $G_\alpha$ and thus also in $G'_\alpha$. Therefore it is easy to embed the paths $P^i$ for $i<m$ in $G'_\alpha$.
\hfill $\Box$ \textit{(Claim 2)}

Suppose for a contradiction that $\Gamma$ is a $\unlhd$-universal graph in $\Forb(T;\leq)$.
For all $\alpha\in A$, we have $G_\alpha\unlhd \Gamma$ by Claim 1 and therefore $G_\alpha\leq \Gamma$ by Claim 2.
We will show that this is impossible.

We define a symmetric binary relation $R$ on the vertices of $\Gamma$ such that $R(v,w)$ if and only if $\Gamma$ contains a subgraph isomorphic to $H_1$ or $H_2$ where $v$ and $w$ play the roles of $x_1$ and $x_2$, respectively.

\smallskip
\noindent\textit{Claim 3:} Let $\alpha\in A$ and suppose that $G_\alpha\subseteq\Gamma$. If $u\in V(G_\alpha)$ is a vertex that is contained in $V(U_\alpha)$ and $v\in V(\Gamma)$ is any vertex, then $R(u,v)$ implies that also $v\in V(U_\alpha)$ and that $u$ and $v$ are adjacent in $U_\alpha$.
\smallskip

\noindent\textit{Proof of Claim 3:}
Let $\alpha\in A$ be given and write $G:=G_\alpha$ and $U:=U_\alpha$. First, note that $G$ contains a subdivision $U'$ of $U$ (see Figure \ref{fig:G_alpha}). We claim that the following is true:
\begin{itemize}
\item[(1)] If $a\neq b\in V(U)$ are adjacent in $\Gamma$, then $a$ and $b$ are also adjacent in $U$ and we have $\alpha(\ell(a,b))=1$ (which means that there is a copy of $H_1$ in $G$ between $a$ and $b$).
\item[(2)] If $a\neq b\in V(U)$, $c\in V(\Gamma)\setminus V(U)$, and $ac,cb\in E(\Gamma)$, then $a$ and $b$ are adjacent in $U$.
\end{itemize}

For the proof of (1), first suppose that $a$ and $b$ are not adjacent in $U$. Then $G+ab\subseteq\Gamma$ contains a copy of $T$ with centre $a$ which is a contradiction: Let $a'$ be the neighbour of $a$ in $U$ which lies on the $a$--$b$ path in $U$. We embed $P^m$ in the copy of $H_1$ or $H_2$ which we inserted for the edge $aa'$ in $G$ and $P^i$ for $i>m$ in $U'+ab$. It is easy to embed $P^i$ for $i<m$ since $a$ has infinite degree in $G$. Thus $a$ and $b$ are adjacent in $U$. Now suppose that $\alpha(\ell(a,b))=2$ and let $\gamma$ be a topological embedding of $H_2$ in $G$ with $\gamma(x_1)=a$ and $\gamma(x_2)=b$. However, then we can find a copy of $T$ in $\Gamma$ with centre $a$ by embedding $P^m$ in the path $a\gamma(y_1)\dots\gamma(y_{p_m})$ and $P^i$ for $i>m$ into $U'+ab$. Again it is easy to embed $P^i$ for $i<m$. The proof of (2) is similar, we only have to choose $U'$ so that $c\notin V(U')$.

For the proof of Claim 3, we begin with the case that $R(u,v)$ witnessed by a subgraph of $\Gamma$ isomorphic to $H_1$ such that $u$ and $v$ play the roles of $x_1$ and $x_2$ in $H_1$, respectively. Then $v$ must lie in $V(U')$ because otherwise we can find a copy of $T$ in $\Gamma$ with centre $u$: We embed all paths $P^i$ with $i>m$ in $U'$ and $P^m$ in one of the infinitely many $u$--$v$ paths which avoids $P^i-u$ for all $i>m$. Therefore, we can assume that $v\in V(U')$ and that we cannot find a different subdivision $U''$ of $U$ in $G$ with $v\notin V(U'')$. Hence $v\in V(U)$. Suppose for a contradiction that $u$ and $v$ are not adjacent in $U$ and choose $U'$ so that $d_{U'}(u,v)>p_m$. However, then we can find a copy of $T$ in $\Gamma$ with centre $u$ by embedding $P^i$ for $m\leq i<k$ in $U'-v$ and $P^k$ in the union of $U'$ with a $u$--$v$ path in $\Gamma$ avoiding $P^i-u$ for all $m\leq i<k$.

Otherwise we have $R(u,v)$ because $H_2\leq\Gamma$, witnessed by an embedding $f$ of $H_2$ in $\Gamma$ such that $f(x_1)=u$ and $f(x_2)=v$. Additionally, we may assume that there is no such embedding of $H_1$ in $\Gamma$. We begin by showing that we can assume that $V(U)$ contains none of the vertices $f(y_1),\ldots,f(y_{p_m-1})$. Suppose that there are $i\neq j$ with $f(y_i),f(y_j)\in V(U)$. Then $uf(y_i),uf(y_j),f(y_i)f(y_j)\in E(U)$ by (1), which is impossible because $U$ is a tree. Therefore at most one of the vertices $f(y_1),\ldots,f(y_{p_m})$ can lie in $V(U)$, without loss of generality $f(y_1),\ldots,f(y_{p_m-1})\notin V(U)$.

If $v\in V(U)$ we can show that $u$ and $v$ are adjacent in $U$, as required in Claim 3: Since $f(y_1)\notin V(U)$ and the edges $uf(y_1)$ and $f(y_1)v$ are contained in $E(\Gamma)$, it follows from (2) that $u$ and $v$ are adjacent in $U$. 
It is left to show that $v\notin V(U)$ is impossible.

Finally, we demonstrate that $v,f(y_1),\ldots,f(y_{p_m-1})\notin V(U)$ leads to a contradiction. We begin by showing that we can always choose $U'$ so that it contains none of the vertices $v,f(y_1),\ldots,f(y_{p_m-1})$. There is only one case in which this is impossible: We must have $H_2\leq G$, witnessed by an embedding $g$ of $H_2$ in $G$ such that $\{v,f(y_1),\ldots,f(y_{p_m-1})\}=\{g(y_1),\ldots,g(y_{p_m})\}$. First we notice that $f(y_{p_m})\in V(U)$, as otherwise we can find a copy of $T$ in $\Gamma$ with centre $g(x_1)$ by embedding $P^m$ in the path $g(x_1)g(y_2)g(y_3)\ldots g(y_{p_m})f(y_{p_m})$ and $P^i$ for $i>m$ into a subdivision of $U$ in $G$ containing $g(y_1)$. Furthermore, since $uf(y_{p_m})\in E(\Gamma)$, (1) implies that there is a vertex $w\in\{u,f(y_{p_m})\}$ which is not an element of $\{g(x_1),g(x_2)\}$. The edges $wf(y_1)$ and $f(y_1)g(x_i)$ are contained in $E(\Gamma)$ for $i=1,2$ and thus (2) implies that $w$ and $g(x_i)$ are adjacent in $U$. However, it follows that $g(x_1)$, $g(x_2)$ and $w$ induce a triangle in the tree $U$. Hence we can assume that $U'$ does not contain $v,f(y_1),\ldots,f(y_{p_m-1})$. But now we can find a copy of $T$ with centre $u$ in the union of $U'$ with the path $uf(y_1)\ldots f(y_{p_m-1})v$ in $\Gamma$, a contradiction.
\hfill $\Box$ \textit{(Claim 3)}

For all $\alpha\in A$ let $\gamma_\alpha$ be an embedding of $G_\alpha$ in $\Gamma$ and define $D_\alpha^i:=\{v\in V(U_\alpha):d_{U_\alpha}(r_\alpha,v)=i\}$ for all $i\in\N$. Since $A$ is uncountable but $\Gamma$ is countable, there exist $\alpha\neq\alpha'\in A$ such that $\gamma_\alpha(r_\alpha)=\gamma_{\alpha'}(r_{\alpha'})$. Without loss of generality, we assume that $G_\alpha,G_{\alpha'}\subseteq \Gamma$ (and not just $G_\alpha,G_{\alpha'}\leq \Gamma$). By Claim 3, we have $D_\alpha^i=D_{\alpha'}^i$ for all $i\in\N$. Since $\alpha\neq\alpha'$, it follows that there are embeddings $f:H_1\to G_\alpha\subseteq\Gamma$ and $g:H_2\to G_{\alpha'}\subseteq\Gamma$ such that $f(x_1)=g(x_1)$ and $f(x_2)=g(x_2)$ (we might have to swap the roles of $\alpha$ and $\alpha'$ for that).

The vertices $g(y_1),\ldots,g(y_{p_m})$ are not contained in any copy of $H_2$ in $G_\alpha$: Suppose for a contradiction that for example $g(y_1)$ is contained in a copy $H'_2$ of $H_2$ in $G_\alpha$. Note that $g(y_1),\ldots,g(y_{p_m})\notin V(U_{\alpha'})=V(U_\alpha)$. Furthermore, $g(y_1)$ is adjacent to $g(x_1)=f(x_1)$ and $g(x_2)=f(x_2)$ and to vertices $y,z$ that play the role of $x_1,x_2$ in $H'_2$. Hence $\{f(x_1),f(x_2)\}\neq \{y,z\}$ and therefore $g(y_1)$ is adjacent to at least three distinct vertices from $V(U_\alpha)$. Then by (2) we can find a triangle in $U_\alpha$, a contradiction.

Let $U'\subseteq G_\alpha$ be a subdivision of $U_\alpha$. Since $g(y_1),\ldots,g(y_{p_m})$ are not contained in any copy of $H_2$ in $G_\alpha$ and $g(y_1),\ldots,g(y_{p_m})\notin V(U_\alpha)$, we can choose $U'$ so that it contains none of the vertices $g(y_1),\ldots,g(y_{p_m})$. However, the union of $U'$ with the path $g(x_1)g(y_1)\ldots g(y_{p_m})$ contains a copy of $T$. Thus a $\unlhd$-universal graph $\Gamma$ in $\Forb(T;\leq)$ cannot exist.
\end{proof}

\begin{proof}[Proof of Theorem \ref{forb-star}]
We combine Theorem \ref{positive} and Theorem \ref{negative}.
\end{proof}

\medskip

\bibliographystyle{amsplain}
\bibliography{bibl}

\end{document}